\documentclass[11pt]{amsart}
\usepackage{amsmath,amssymb,amsfonts,amsthm,amscd,indentfirst}

\usepackage[dvipdfm]{hyperref}

\numberwithin{equation}{section}

\newtheorem{prop}{Proposition}[section]
\newtheorem{theo}[prop]{Theorem}

\newtheorem{coro}[prop]{Corollary}
\newtheorem{rema}[prop]{Remark}

\newcommand {\eqdef }{\ensuremath {\stackrel {\mathrm {def
}}{=}}}
\newcommand \xLongLeftRightArrow[2][]{ %
\ext@arrow 0055{\LongLeftRightArrowfill@ }{#1}{#2}}
\def\LongLeftRightArrowfill@ {%
\arrowfill@ \Leftarrow \Relbar \Rightarrow }

\newcommand \xlongleftrightarrow[2][]{
\ext@arrow 0055{\longleftrightarrowfill@}{#1}{#2}}
\def\longleftrightarrowfill@ {
\arrowfill@ \leftarrow \relbar \rightarrow }

\def\begeq{\begin{equation}}
\def\endeq{\end{equation}}

\def\Dint{\displaystyle\int}

\begin{document}
\title{Monotonicity formulas
under rescaled Ricci flow}
\author{Jun-Fang Li}
\address{Department of Mathematics\\
         McGill University\\
         Montreal, Quebec. H3A 2K6, Canada.}
\email{jli@math.mcgill.ca}
\thanks{Research of the author is
supported in part by a CRM fellowship.}

\date{}
\maketitle

\begin{abstract}
In this short notes, we discuss monotonicity formulas under various
rescaled versions of Ricci flow. The main result is Theorem
\ref{theo rescaled}.
\end{abstract}


\section{{\bf Functionals $\mathcal{W}_{ek}$ from rescaled Ricci flow point of view}}
This is the research notes when the author wrote \cite{L1}. In the
first section, we discuss the relation between functionals
$\mathcal{W}_{ek}(g,f,\tau)$ and rescaled Ricci flow.

In {\em{\bf Theorem} 4.2} \cite{L1}, we have defined functionals
$\mathcal{F}_k( g ,f) = \Dint_M (k R + |\nabla f|^2)e^{-f}d\mu$
where $k\ge 1$ and derived the first variational formula under a
coupled system (1.2) as following
\begin{equation}\label{Lia 01}
\begin{array}{rll}
\frac{d}{d t}\mathcal{F}_k( g _{ij},f)=2(k-1)\Dint_M
|Rc|^2e^{-f}d\mu+2\Dint_M |R_{ij}+\nabla_i\nabla_j f|^2e^{-f}d\mu
\ge 0.
\end{array}
\end{equation}
In particular, this yields the following theorem.

\begin{theo}\label{}({\bf Theorem} 5.2 in \cite{L1})
On a compact Riemannian manifold $(M,g(t))$, where $g(t)$ satisfies
the Ricci flow equation for $t\in [0,T)$, the lowest eigenvalue
$\lambda_k$ of the operator $-4\Delta + kR$ is nondecreasing under
the Ricci flow. The monotonicity is strict unless the metric is
Ricci-flat.
\end{theo}

\begin{rema}
Formula (\ref{Lia 01}) previously was found by physicists
independently \cite{OSW}. We thank professor E. Woolgar for giving
us the reference.
\end{rema}

Based on these observations, to classify expanding Ricci breathers,
we have introduced a family of new functionals $\mathcal{W}_{ek}$
which has monotonicity properties modeled on expanders. There is a
closed relation between functional $\mathcal{F}_k$ and
$\mathcal{W}_{ek}$ connected by rescaled Ricci flow. Actually, this
was one of the motivations for us to introduce $\mathcal{W}_{ek}$
(see Remark 6.2 in
\cite{L1}).\\

It is well-known that there is a one-to-one correspondence between
Ricci flow and rescaled Ricci flow. Suppose $g(\cdot,t)$ is a
solution of Ricci flow equation
\begin{equation}\label{Ricci flow equ}
\begin{array}{rll}
\frac{\partial}{\partial t}g_{ij}=-2R_{ij}.
\end{array}
\end{equation}

For any given  function $s(t)$, if
$\varphi(t)=\frac{1}{1-\frac{2}{n}\int^t_0s(t)dt}$, and $\bar
t=\int^t_0\varphi(t)dt$, then $\bar g(\cdot,\bar
t)=\varphi(t)g(\cdot,t)$ solves the rescaled Ricci flow equation
\begin{equation}\label{Scaled Ricci flow equ}
\begin{array}{rll}
\frac{\partial}{\partial \bar t}\bar g_{ij}=-2(\bar
R_{ij}-\frac{s}{n}\bar g_{ij})
\end{array}
\end{equation}
and
\begin{equation}\label{Relation thm1 equ1}
\boxed{
\begin{array}{rll}
\mathcal{F}_k=& \Dint_M
\big(kR+\Delta f\big)e^{-f}d\mu\\
\frac{\partial g _{ij}}{\partial t} =& -2R_{ij}\\
\frac{\partial f}{\partial t} =& -\Delta f +|\nabla f|^2 - R\\
\end{array}
}  \stackrel{\varphi g=\bar g}{\Longleftrightarrow}
 \boxed{
\begin{array}{rll}
{\mathcal{F}}_k=& \varphi\Dint_M
\big(k\bar R+\bar \Delta \bar f\big)e^{-\bar f}d\bar \mu\\
\frac{\partial \bar g _{ij}}{\partial \bar t} =& -2(\bar
R_{ij}-\frac{s}{n}\bar g
_{ij})\\
\frac{\partial \bar f}{\partial \bar t} =& -\bar \Delta \bar f +|\bar \nabla \bar f|^2 - \bar R+s\\
\end{array}
}
\end{equation}

Furthermore, (\ref{Lia 01}) is equivalent to
\begin{equation}\label{SRF mono 1}
\begin{array}{rll}
\frac{d}{d \bar t}\mathcal{\bar F}_k=-\frac{2s}{n}\mathcal{\bar
F}_k+2(k-1)\Dint_M |\bar {Rc}|^2e^{-\bar f}d\mu+2\Dint_M |\bar
R_{ij}+\bar \nabla_i\bar \nabla_j \bar f|^2e^{-\bar f}d\bar \mu,
\end{array}
\end{equation}
where we use $\bar{\mathcal{ F}}_k$ to denote $\mathcal{ F}_k(\bar
{g}_{ij},\bar f)=\Dint_M \big(k\bar R+\bar \Delta \bar
f\big)e^{-\bar f}d\bar \mu$.

 Notice that this rescaled Ricci flow is a generalized version
which includes Hamilton's normalized Ricci flow as a special case.

When $s=0$, this is the monotonicity formula under Ricci flow
without rescale. When $s<0$, we have the monotonicity under rescaled
Ricci flow. By the same proof of {\bf Theorem} 5.2 in \cite{L1}, we
have the following monotonicity formula for the lowest eigenvalues
under the rescaled Ricci flow.

\begin{prop}\label{prop SRF mono 1}
On a compact Riemannian manifold $(M,g(t))$, if $g(t)$ satisfies the
 rescaled Ricci flow equation (\ref{Scaled Ricci flow
equ}) for $t\in [0,T)$ with $s\le0$, then the lowest eigenvalue
$\lambda$ of the operator $-4\Delta + kR$ ($k\ge 1$) is
nondecreasing under the rescaled Ricci flow.
\end{prop}

Compare Proposition \ref{prop SRF mono 1} with Theorem \ref{Lia 01},
the main difference is Proposition \ref{prop SRF mono 1} fails to
classify the steady state of the lowest eigenvalues which is crucial
for applications of monotonicity formulas in general. Namely when
the monotonicity is not strict, it yields no information. Thus one
cannot apply this formula directly to classify steady or expanding
breathers while the functionals $\mathcal{W}_{ek}$ we previously
introduced served well for this purpose.\\

From the point of view of rescaled Ricci flow, it is natural for us
to introduce functionals $\mathcal{W}_{ek}$. Using (\ref{SRF mono
1}), we have
\begin{equation}\label{SRF mono 2}
\begin{array}{rll}
\frac{d}{d \bar t}\mathcal{\bar F}_k=\frac{2s}{n}\mathcal{\bar
F}_k-\frac{2ks^2}{n}+2(k-1)\Dint_M |\bar Rc-\frac{s}{n}\bar
g|^2e^{-\bar f}d\mu+2\Dint_M |\bar R_{ij}+\bar \nabla_i\bar \nabla_j
\bar f-\frac{s}{n}\bar g|^2e^{-\bar f}d\bar \mu.
\end{array}
\end{equation}

In the following, if we choose $s=constant$, and define
$\mathcal{\bar W}_k\eqdef\mathcal W_k(\bar g,\bar
f)=e^{-\frac{2s}{n}\bar t}(\mathcal {\bar F}_k -ks)$, then
\begin{equation}\label{SRF mono 3}
\begin{array}{rll}
\frac{d}{d \bar t}\mathcal{\bar W}_k=e^{-\frac{2s}{n}\bar
t}\bigg[2(k-1)\Dint_M |\bar Rc-\frac{s}{n}\bar g|^2e^{-\bar
f}d\mu+2\Dint_M |\bar R_{ij}+\bar \nabla_i\bar \nabla_j \bar
f-\frac{s}{n}\bar g|^2e^{-\bar f}d\bar \mu\bigg].
\end{array}
\end{equation}

We can write out $\bar t(t)$ and $\varphi(t)$ explicitly. If we
denote $\tau(t)=-\frac{2n}{s\varphi}$, then
$\tau=(-\frac{2n}{s}+t)$. Now back to the corresponding unrescaled
Ricci flow system, we have established
\begin{equation}\label{SRF mono 4}
\begin{array}{rll}
\mathcal
W_k(g,f,\tau)\eqdef\tau^2\Dint_M[k(R+\frac{n}{2\tau})+\Delta
f]e^{-f}d\mu
=\mathcal W_k(\bar g,\bar f)\\
\end{array}
\end{equation}

\begin{equation}\label{SRF mono 5}
\begin{array}{rll}
&\frac{d}{d t}\mathcal{ W}_k={\frac{d}{d\bar t}\mathcal{ W}_k}\cdot{\frac{d\bar t}{dt}}\\
=&2(k-1)\tau^2\Dint_M | {Rc}-\frac{1}{2\tau} g|^2e^{-
f}d\mu+2\tau^2\Dint_M | R_{ij}+ \nabla_i \nabla_j f-\frac{1}{2\tau}
g|^2e^{- f}d \mu\\
\end{array}
\end{equation}
This is the same formula of $\mathcal{W}_{ek}$ that we obtained in
\cite{L1}. The advantage of this formula is, clearly, we can read
that the steady states of the above functionals are Einstein
manifolds. One can further apply this property to classify expanding
breathers, see \cite{L1}. As an application to lowest eigenvalues,
we have

\begin{theo}\label{theo SRF mono 1}
On a compact Riemannian manifold $(M,g(t))$, where $g(t)$ satisfies
Ricci flow equation (\ref{Ricci flow equ}) for $t\in [0,T)$, if
$\lambda$ denotes the lowest eigenvalue of the operator $-4\Delta +
kR$, then $\tau^2(\lambda+k\frac{n}{2\tau})$ is nondecreasing under
 Ricci flow with $\frac{d \tau}{dt}=1$. The monotonicity is strict unless the metric is Einstein.
 (See Remark \ref{different k}.)
\end{theo}

Back to the rescaled Ricci flow, we have the twin theorem of the
above.

\begin{theo}\label{theo SRF mono 2}
On a compact Riemannian manifold $(M,g(t))$, where $g(t)$ satisfies
the rescaled Ricci flow equation (\ref{Scaled Ricci flow equ}) for
$t\in [0,T)$ with $s=constant$, if $\lambda$ denotes the lowest
eigenvalue of the operator $-4\Delta + kR$, then
$e^{-\frac{2s}{n}t}(\lambda-ks)$ is nondecreasing under the rescaled
Ricci flow. The monotonicity is strict unless the metric is
Einstein. (See Remark \ref{different k}.)
\end{theo}

\begin{rema}
One can find applications of the above theorems in K\"ahler Ricci
flow, in which case, one can choose $s$ to be constants.
\end{rema}

\section{\bf{Monotonicity of lowest eigenvalues under rescaled Ricci flow}}
A direct consequence of (\ref{SRF mono 2}) yields the following
monotonicity property.
\begin{theo}\label{theo rescaled}
On a compact Riemannian manifold $(M^n,g(t))$, where $g(t)$
satisfies the rescaled Ricci flow equation (\ref{Scaled Ricci flow
equ}) for $t\in [0,T)$, we denote $\lambda(t)$ to be the lowest
eigenvalue of the operator $-4\Delta + kR$ ($k\ge 1$)at time $t$.
Assume that there exists a function $\varphi(x,t)\in
C^{\infty}(M^n\times [0,T))$, such that
$s(t)=\frac{\int_M(kR+|\nabla \varphi|^2)e^{-\varphi}d\mu}{k\int
e^{-\varphi}d\mu}$, and also $\lambda(t)$ is a $C^1$ family of $t$,
then $\lambda$ is nondecreasing under the rescaled Ricci flow,
provided $s\le0$. The monotonicity is strict unless the metric is
Einstein.
\end{theo}

\begin{proof}
To simplify notations, we remove all the $\bar{\quad}$ in (\ref{SRF
mono 2}). Under the hypothesis of the theorem, proceed as before, we
have
\begin{equation}\label{}
\begin{array}{rll}
\frac{d\lambda}{d t}=\frac{2s}{n}(\lambda-ks)+2(k-1)\Dint_M |
{Rc}-\frac{s}{n}g|^2e^{- f}d\mu+2\Dint_M | R_{ij}+ \nabla_i \nabla_j
 f-\frac{s}{n} g|^2e^{- f}d \mu,
\end{array}
\end{equation}
where $e^{-\frac{f}{2}}$ is the eigenfunction of $\lambda$. We know
$s\le0$, and by definitions of $\lambda$, $s$,  we have $\lambda\le
ks$ and $\frac{2s}{n}(\lambda-ks)\ge 0$. Hence,
\begin{equation}\label{}
\begin{array}{rll}
\frac{d\lambda}{d t}\ge2(k-1)\Dint_M | {Rc}-\frac{s}{n}g|^2e^{-
f}d\mu+2\Dint_M | R_{ij}+ \nabla_i \nabla_j
 f-\frac{s}{n} g|^2e^{- f}d \mu.
\end{array}
\end{equation}
This completes the proof of the theorem.
\end{proof}

\begin{rema}\label{}
(\ref{SRF mono 2}) can be obtained by direct computations instead of
using rescaling of the metrics.
\end{rema}

\begin{rema}\label{different k}
In the proof of Theorems \ref{theo SRF mono 1}, \ref{theo SRF mono
2}, and \ref{theo rescaled}, when $k=1$, one actually first prove
that it is a compact gradient expanding (or steady) Ricci soliton,
then by the known classification theorems, we know it must be
Einstein.
\end{rema}

In case of Hamilton's normalized Ricci flow, where
$s=\frac{\int_MRd\mu}{\int_Md\mu}$ is the average total scalar
curvature, and one can choose $\varphi(t)=\ln Vol(M^n)$, we have

\begin{coro}\label{coro Normalized}
On a compact Riemannian manifold $(M^n,g(t))$, where $g(t)$
satisfies the normalized Ricci flow equation of Hamilton for $t\in
[0,T)$, if $\lambda(t)$ denotes the lowest eigenvalue of the
operator $-4\Delta + kR$ ($k\ge 1$)at time $t$, assume $\lambda(t)$
is a $C^1$ family of $t$, then $\lambda$ is nondecreasing under the
normalized Ricci flow, provided the average total scalar curvature
is nonpositive. The monotonicity is strict unless the metric is
Einstein.
\end{coro}

\begin{rema}
The above monotonicity property of lowest eigenvalues under
normalized Ricci flow is dimensionless and work for all $k\ge 1$ and
the case of $k>1$ can be used to classify compact steady or
expanding Ricci breathers directly, which, in fact, yields another
proof for {\bf Corollary} 8.1 in \cite{L1}. See references of
related results in \cite{H1}, \cite{I}, \cite{P}, \cite{BM},
\cite{C0} and \cite{L1}.
\end{rema}

\begin{rema}
Similar results of Corollary \ref{coro Normalized} appeared recently
in \cite{C}.
\end{rema}

In the special case $k=1$, recall Perelman's $\bar\lambda=\lambda
V^\frac{2}{n}$ invariant \cite{P}. Since normalized Ricci flow
preserves the volume and $\bar\lambda$ is a scale invariant, the
monotonicity of lowest eigenvalues under normalized Ricci flow is
equivalent to the monotonicity of $\bar\lambda$ under Ricci flow.
Naturally, one could extend this definition and introduce a new
scale invariant constant $\bar\lambda_k\eqdef
\lambda_kV^\frac{2}{n}$ which also has monotonicity properties along
(normalized) Ricci flow and vanishes if and only if on an expanding
{\em Einstein} manifold (instead of expanding gradient solitons in
the $k=1$ case).

\begin{rema}
There is some fundamental relation between $\bar\lambda$ and the
Yamabe constant $\mathcal Y$ discussed in \cite{AIL} and the
references therein. Similar results can also be obtained for
$\bar\lambda_k$.
\end{rema}

\section{\bf Acknowledgement}
We would like to thank professor Pengfei Guan for his encouragement
and professor Xiaodong Cao for helpful discussions.

\end{document}